\documentclass[a4paper,12pt]{article}
\usepackage{geometry,latexsym,amssymb}
\usepackage[latin5]{inputenc}
\usepackage{enumerate}
\usepackage{enumitem}
\usepackage[usenames]{color}
\usepackage{graphicx}
\usepackage{enumerate}
\usepackage{rotating}
\usepackage{multirow}
\usepackage{cite}
\usepackage{amsthm}
\usepackage{tikz}
\usetikzlibrary{matrix,arrows}
\usepackage{amsmath,amscd}
\usepackage{calc}
\usepackage{ifthen}
\usepackage[hang,labelsep=period]{caption}
\usepackage{fancyhdr}

 \newtheorem{thm}{Theorem}[section]
 \newtheorem{cor}[thm]{Corollary}
 \newtheorem{lem}[thm]{Lemma}
 
 \theoremstyle{definition}
 
 \theoremstyle{remark}
 
 \newtheorem*{ex}{Example}
 \numberwithin{equation}{section}

\begin{document}

\begin{center}
{\Large \textbf{A NEW APPROACH TO A THEOREM OF ENG}} \vspace*{0.5cm}
\end{center}

\begin{center}
HASAN ARSLAN$^{*}$ \\
$^{*}${\small {\textit{Department of  Mathematics, Faculty of Science, Erciyes University, 38039, Kayseri, Turkey}}}\\
{\small {\textit{E-mail: hasanarslan@erciyes.edu.tr}}}\\[0pt]
\end{center}
\vspace*{0.3cm}

ABSTRACT. The main aim of this work is to give a case-free algebraic proof for a theorem of Eng on the Poincar\'e polynomial of parabolic quotients of finite Coxeter groups evaluated at -1. 

\vspace*{0.3cm}

\section{Introduction}

Let $(W,S)$ be a finite Coxeter system. For any subset $J$ of $S$, if the simple reflections in $J$ generate $W_{J}$, then $W_{J}$ is called \textit{a standard parabolic subgroup} of $W$. Denote by $X_{J}$ the set of distinguished left coset representatives $W_{J}$ in $W$ and let $w_{0}$ be the longest element of $W$. We use the same notations $P_{J}$ and $N_{J}$ as in \cite {br8} for the number of elements in $X_{J}$ of even and odd length, respectively and let $D_{J}=P_{J}-N_{J}$. In [\cite{br2}, Theorem 1], O. Eng proved the following theorem about -1 phenomenon.
\vspace*{0.3cm}

\begin{thm}\label{Eng}
For any $J \subseteq S$, we have
\begin{equation}\label{-1 phnmn}
    \sum_{w \in X_{J}} (-1)^{l(w)}= |\{ w \in X_{J}: w_{0}wW_{J}=wW_{J} \}|,
\end{equation}
where $l$ denotes the length function on $W$.
\end{thm}
In \cite {br8}, L. Tan showed the relation $D_{J}=\sum_{w \in X_{J}} (-1)^{l(w)}$. O. Eng proved this theorem by using case by case techniques. In \cite{br5}, V. Reiner gave a case-free geometric proof of Theorem \ref{Eng}. Then in \cite{br6}, V. Reiner, D. Stanton, D. White presented the first case-free algebraic proof of the theorem above. To give another proof of Theorem \ref{Eng} algebraically, we use a different method than that of V. Reiner, D. Stanton, D. White. Our approach depends more on the structure of the descent algebra of a finite Coxeter group introduced by L. Solomon in \cite{br7}. Our proof is new and avoids case by case considerations. Set $x_{J}=\sum_{w \in X_{J}}w$ for any subset $J$ of $S$. Then $\{ x_{J}~|~ J \subseteq S \}$ forms a basis for a subalgebra of the group algebra $\mathbb{Q}W$ called \textit{the descent algebra} of $W$. We denote by $\sum(W)$ the descent algebra corresponding to $W$. The ascent set of $w$ is defined by
$$\mathcal{R}(w)=\{s \in S : l(ws)>l(w)\}.$$
For any subset $I$ of $S$, put $Y_{I}=\{w \in W : \mathcal{R}(w)=I\}$. We consider the element $y_{I}=\sum_{w \in Y_{I}}w$ in $\mathbb{Q}W$. Then
$$x_{I}=\sum_{I \subseteq J}y_{J}$$
and by M\"obius inversion formula
$$y_{I}=\sum_{I \subseteq J}(-1)^{|J-I|}x_{J}.$$
Thus the set $\{y_{I} : I \subseteq S\}$ is a basis of $\sum(W)$, see \cite {br7}.
In \cite{br7}, L. Solomon also defined an algebra map from $\sum(W)$ to $\mathbb{Q}\textrm{Irr}W$ as follows:
$$\Phi : \sum(W) \rightarrow \mathbb{Q}\textrm{Irr}(W), x_{J} \mapsto Ind_{W_{J}}^{W}1_{W_{J}},$$
where $\mathbb{Q}\textrm{Irr}W$, $1_{W_{J}}$ and $Ind_{W_{J}}^{W}1_{W_{J}}$ denote the algebra generated by irreducible characters of $W$, the trivial character of $W_{J}$ and the permutation character of $W_{J}$ in $W$, respectively. Taking into account the sign character of $W$, which is defined as $\varepsilon : W \rightarrow \mathbb{N},~\varepsilon(w)=(-1)^{l(w)}$ (see \cite{br3}), it is well-known from \cite{br7} that $y_{\emptyset}=w_{0} \in \sum(W)$ and
\begin{equation}\label{epsilon}
\Phi(w_{0})=\varepsilon.
\end{equation}
The sign character $\varepsilon$ of $W$ is actually irreducible and equals to the Steinberg character of $W$, which is given by the formula $St_{W}=\sum_{J \subseteq S}(-1)^{|J|}\Phi(x_{J})$ due to \cite{br11}.
In [\cite{br1}, Main Theorem], D. Blessenohl, C. Hohlweg, M. Schocker showed the symmetry property for the descent algebra $\sum(W)$, that is,
\begin{equation}\label{symmetry}
    \Phi(x)(y)=\Phi(y)(x)
\end{equation}
for all $x, y \in \sum(W)$.

Let $I, J \subseteq S$. We write $I \sim_{W} J$ if $I$ and $J$ are $W$-conjugate, that is, there is a $w \in W$ such that $I=wJw^{-1}$ .

\section{\textit{Proof of Theorem 1.1}}

When we extend linearly the sign character $\varepsilon$ of $W$ to the group algebra $\mathbb{Q}W$ and use the equation (\ref{epsilon}), (\ref{symmetry}), then we conclude that
\begin{align*}
\sum_{w \in X_{J}} (-1)^{l(w)}&=\sum_{w \in X_{J}}\varepsilon(w)=\varepsilon(\sum_{w \in X_{J}}w) \\
&=\varepsilon(x_{J})=\Phi(w_{0})(x_{J})=\Phi(x_{J})(w_{0})\\
&=Ind_{W_{J}}^{W}1_{W_{J}}(w_{0})=|\{ w \in X_{J}: w_{0}wW_{J}=wW_{J} \}|.
\end{align*}
Therefore, we obtain the equation (\ref{-1 phnmn}) and so we complete a case-free algebraic proof of the Theorem \ref{Eng}. Consequently, we can write $D_{J}=\Phi(x_{J})(w_{0})$.
\begin{cor}
Let $I, J \subseteq S$. If $I \sim_{W} J$, then we have $D_{I}=D_{J}$.
\end{cor}
\begin{proof}
The radical of $\sum(W)$ is $Ker\Phi=<x_{I}-x_{J} : I \sim_{W} J>$ from \cite{br7}. Since $\Phi$ is an algebra morphism, then we get $\Phi(x_{I})=\Phi(x_{J})$. By considering the proof of Theorem \ref{Eng} given above, one has
\begin{equation*}
    D_{I}=\Phi(x_{I})(w_{0})=\Phi(x_{J})(w_{0})=D_{J}.
\end{equation*}
\end{proof}
Another perspective for the proof of Corollary 2.1 is that when $I$ conjugate to $J$ under $W$, the corresponding sets of the distinguished coset representatives $X_{I}$ and $X_{J}$ are $W$-pointwise conjugate to each other due to \cite{br4}. Therefore, we have $D_{I}=D_{J}$.
\begin{cor}\label{sum}
For any finite Coxeter system $(W,S)$, we have
$$\sum_{J \subseteq S}(-1)^{|J|}D_{J}=(-1)^{|T|},$$
where $T$ denotes the set of reflections of $W$, that is, $T=\bigcup_{w \in W}wSw^{-1}$.
\end{cor}

\begin{proof}
To prove Corollary \ref{sum} observe that image of the longest element $w_{0}$ under the sign character $\varepsilon$ of $W$ is $\varepsilon(w_{0})=(-1)^{l(w_{0})}=(-1)^{|T|}$. Since $w_{0}=\sum_{J \subseteq S}(-1)^{|J|}x_{J}$, we have
$$\varepsilon(w_{0})=\sum_{J \subseteq S}(-1)^{|J|}\varepsilon(x_{J})=\sum_{J \subseteq S}(-1)^{|J|}\Phi(x_{J})(w_{0})=\sum_{J \subseteq S}(-1)^{|J|}D_{J},$$
using the fact that $\Phi(x_{J})(w_{0})=D_{J}$.
\end{proof}
One can observe that Corollary  \ref{sum} is actually a special case of Proposition 1.11 given in \cite{br3}.
\begin{ex}
We consider the Coxeter system $(W_{n},S_{n})$ of type $B_{n}$. In [\cite{br8}, Theorem B], L. Tan proved that $D_{J}=0$ for every proper parabolic subgroup of $W_{n}$. This statement is also seen from the fact that there is no $w \in X_{J}$ such that $ww_{0}w^{-1} \in W_{J}$ for every proper subset $J$ of $S_{n}$. Therefore,
$$\sum_{J \subseteq S_{n}}(-1)^{|J|}D_{J}=(-1)^{n}.$$
\end{ex}
The following result due to L. Tan \cite{br8} is easily derived from Lemma 2.1 given in \cite{br0} and Theorem \ref{Eng}.
\begin{lem}[Tan \cite{br8}]
Suppose $I \subseteq J \subseteq S$. Then $W_{J}=X_{I}^{J}W_{I}$ and $D_{SI}=D_{SJ}D_{JI}$, where $X_{I}^{J}=X_{I} \cap W_{J}$.
\end{lem}
It is well-known from Lemma 2.1 in \cite{br0} that for $I \subseteq J \subseteq S$, we have $x_{I}=x_{J}x_{I}^{J}$. Since $\Phi$ is an algebra morphism, then we get
$$\Phi(x_{I})=\Phi(x_{J}) \Phi(x_{I}^{J}).$$
Taking Theorem \ref{Eng} into consideration we conclude that the relation $D_{SI}=D_{SJ}D_{JI}$ in the sense of \cite{br8}.
\begin{cor}
For any finite Coxeter group $W$, we have
$$\sum_{w \in W}(-1)^{l(w)}=0.$$
\end{cor}
\begin{proof}
Since the image of $x_{\emptyset}$ under $\Phi$ gives the regular character of $W$ and the equation (\ref{symmetry}), then we conclude
$$\sum_{w \in W}(-1)^{l(w)}=\varepsilon(\sum_{w \in W}w)=\varepsilon(x_{\emptyset})=\phi(w_{0})(x_{\emptyset})=\phi(x_{\emptyset})(w_{0})=0.$$.
\end{proof}

\end{document}